\documentclass[11pt,a4paper,reqno]{amsart}
\usepackage[latin1]{inputenc}
\usepackage{amsmath}
\usepackage{amsfonts}
\usepackage{amssymb}
\usepackage{amsthm}
\usepackage{anysize}
\usepackage{enumitem}
\usepackage{amscd}
\usepackage{hyperref}
\usepackage[square, comma, numbers, sort&compress]{natbib}
\usepackage{indentfirst}
\marginsize{3.0cm}{3.0cm}{3.0cm}{3.0cm}
\setenumerate{label={\normalfont(\arabic*)}}
\newcommand{\coloneqq}{\mathrel{\mathop:}=}
\newcommand{\form}[1]{{\langle #1 \rangle }}
\newcommand{\pfister}[1]{{\langle \! \langle #1 \rangle \! \rangle}}
\newcommand{\mydim}[1]{{\mathrm{dim}( #1)}}
\newcommand{\windex}[1]{{\mathfrak{i}_W(#1)}}
\newcommand{\witti}[2]{{\mathfrak{i}_{#1}(#2)}}
\newcommand{\wittj}[2]{{\mathfrak{j}_{#1}(#2)}}
\newcommand{\anispart}[1]{{#1_{\mathrm{an}}}}
\newcommand{\chow}[1]{{\mathrm{Ch}(#1)}}

\newtheorem{theorem}{Theorem}[section]

\newtheorem{lemma}[theorem]{Lemma}

\newtheorem{proposition}[theorem]{Proposition}
\newtheorem{corollary}[theorem]{Corollary}
\newtheorem{conjecture}[theorem]{Conjecture}

\theoremstyle{definition}

\newtheorem{example}[theorem]{Example}

\newtheorem{question}[theorem]{Question}

\theoremstyle{remark}
\newtheorem{remark}[theorem]{Remark}
\newtheorem{remarks}[theorem]{Remarks}

\numberwithin{equation}{section}
\begin{document}

\title[]{Hyperbolicity and near hyperbolicity of quadratic forms over function fields of quadrics}
\author{Stephen Scully}
\address{Department of Mathematical and Statistical Sciences, University of Alberta, Edmonton AB T6G 2G1, Canada}
\email{stephenjscully@gmail.com}


\maketitle

\begin{abstract} Let $p$ and $q$ be anisotropic quadratic forms over a field $F$ of characteristic $\neq 2$, let $s$ be the unique non-negative integer such that $2^s < \mydim{p} \leq 2^{s+1}$, and let $k$ denote the dimension of the anisotropic part of $q$ after scalar extension to the function field $F(p)$ of $p$. We conjecture that $\mydim{q}$ must lie within $k$ of a multiple of $2^{s+1}$. This can be viewed as a direct generalization of Hoffmann's Separation theorem. Among other cases, we prove that the conjecture is true if $k<2^{s-1}$. When $k=0$, this shows that any anisotropic form representing an element of the kernel of the natural restriction homomorphism $W(F)\rightarrow W(F(p))$ has dimension divisible by $2^{s+1}$. \end{abstract}

\section{Introduction} \label{secintroduction} Many of the central problems in the algebraic theory of quadratic forms seem to demand an investigation of the behaviour of quadratic forms under scalar extension to function fields of quadrics. This was already apparent from the early beginnings of the subject, with function fields of Pfister quadrics being used in an essential way to prove the fundamental Arason-Pfister Hauptsatz of 1971 (\cite{ArasonPfister}). Arason and Pfister's work led to the foundational papers of Elman-Lam (\cite{ElmanLam}) and Knebusch (\cite{Knebusch1},\cite{Knebusch2}), after which properties of function fields of quadrics have been systematically studied and exploited in many of the major developments in the theory; notable examples include the proof of the Milnor conjectures (\cite{Voevodsky},\cite{OVV}) and essentially all advances on Kaplansky's problem concerning the possible $u$-invariants of fields (\cite{Merkurjev},\cite{Izhboldin2},\cite{Vishik3}).

An important problem in this context is the following: Let $p$ and $q$ be anisotropic quadratic forms of dimension $\geq 2$ over a field $F$ of characteristic $\neq 2$, and let $F(p)$ denote the function field of the projective quadric $\lbrace p=0 \rbrace$. Under what circumstances does $q$ become isotropic over $F(p)$? While this problem seems to be rather intractable in general, the extraction of partial information is already enough for non-trivial applications. As a result, it has been of great interest here to identity general constraints coming from the basic invariants of the forms $p$ and $q$. A key breakthrough in this direction was made in \cite{Hoffmann1}, where Hoffmann determined the constraints coming from the simplest invariants of all, namely, the dimensions of $p$ and $q$. More precisely, Hoffmann's ``Separation theorem'' asserts that if $\mydim{q} \leq 2^s < \mydim{p}$ for some integer $s$, then $q$ remains anisotropic over $F(p)$, and this is optimal, in the sense that isotropy can occur in all other cases. An important generalization of Hoffmann's theorem taking into account one further invariant of $p$ was later given by Karpenko and Merkurjev in \cite{KarpenkoMerkurjev}.

The present article is concerned with the following variant of the above question: To what \emph{extent} can the form $q$ become isotropic over $F(p)$? Formally, the extent to which a quadratic form is isotropic is measured by its \emph{Witt index}, i.e., the maximal dimension of a totally isotropic subspace of the vector space on which it is defined. In analogy with Hoffmann's result, one can ask the following: What constraints do the dimensions of $p$ and $q$ impose on the Witt index $\windex{q_{F(p)}}$ of $q$ after extension to $F(p$)? Here the results of \cite{Hoffmann1} and \cite{KarpenkoMerkurjev} only go so far. In particular, they are essentially vacuous in the case where $\mydim{p} \leq \mydim{q}/2$, whereas the problem itself is non-trivial in all dimensions. A similar criticism applies to the main result of \cite{Scully1}, which gives a certain refinement the result of Karpenko and Merkurjev. The purpose of this paper is to propose the following conjectural answer to the above problem:

\begin{conjecture} \label{conjmain} Suppose that $p$ and $q$ are anisotropic quadratic forms of dimension $\geq 2$ over $F$. Let $s$ be the unique non-negative integer such that $2^s < \mydim{p} \leq 2^{s+1}$, and let $k= \mydim{q} - 2\windex{q_{F(p)}}$ $($i.e., $k$ is the dimension of the anisotropic part of $q_{F(p)})$. Then
$$\mydim{q} = a2^{s+1} + \epsilon $$
for some non-negative integer $a$ and some integer $-k \leq \epsilon \leq k$. \end{conjecture}

Again, it is not hard to see that this is optimal, to the extent that there can be no further gaps in the possible values of $\mydim{q}$ determined by $\windex{q_{F(p)}}$ and $\mydim{p}$ alone (see Example \ref{exoptimality} below). In particular, the statement include's Hoffmann's theorem:

\begin{example} Suppose that $\mydim{q} \leq 2^s$, so that we have separation of dimensions by a power of $2$. The reader will easily confirm that, in this case, it is only possible to express $\mydim{q}$ in the suggested way if $k=\mydim{q}$, i.e., if $q$ remains anisotropic over $F(p)$. \end{example}

Conjecture \ref{conjmain} is vacuously true if $k \geq 2^s-1$, so we are interested here in the case where $k \leq 2^s - 2$. In other words, beyond the Separation theorem, we are looking at the situation in which $\windex{q_{F(p)}}$ is ``large'' (which explains the title of the article). Loosely speaking, the conjecture asserts that the more isotropic $q$ becomes over $F(p)$, the closer its dimension should be to being divisible by $2^{s+1}$. Our main result is the following:

\begin{theorem} \label{thmmain} Conjecture \ref{conjmain} is true when $k<2^{s-1}$.  \end{theorem}

In particular, the conjecture is true in the extreme case where $k=0$, which translates as follows:

\begin{corollary} \label{thmWittkernel} Let $p$ and $q$ be anisotropic quadratic forms of dimension $\geq 2$ over $F$, and let $s$ be the unique non-negative integer such that $2^s < \mydim{p} \leq 2^{s+1}$. If $q$ becomes hyperbolic over $F(p)$, then $\mydim{q}$ is divisible by $2^{s+1}$. \end{corollary}

Put another way, Corollary \ref{thmWittkernel} says the following: Let $W(F(p)/F)$ denote the kernel of the natural scalar extension homomorphism $W(F)\rightarrow W(F(p))$ on Witt groups. If $q$ represents an element of $W(F(p)/F)$, then $\mydim{q}$ is divisible by $2^{s+1}$. A few low-dimensional cases aside (see, e.g., \cite{Fitzgerald2}), this seems to have been unknown, even conjecturally.\footnote{Little seems to have been known beyond the fact that $\mydim{q} \geq 2^{s+1}$ in this case, which is an easy consequence of the Cassels-Pfister subform theorem (see \cite[Lem. 4.5]{Knebusch1}).}. In fact, one can go rather further here, and show that all higher Witt indices of $q$, except possibly the last, are divisible by $2^{s+1}$ in this case (Theorem \ref{thmdivisibilityofindices}). This can be viewed as a precise generalization of a well-known theorem of Fitzgerald (\cite{Fitzgerald1}) on the low-dimensional part of $W(F(p)/F)$ (see \ref{subsecfitz} below). Although the proof of the Arason-Pfister Hauptsatz already drew serious attention to the problem of determining Witt kernels of function fields of quadrics, very little progress has been made over the last 40 years, with Fitzgerald's theorem being among the few highlights. Our computations perhaps raise some new questions here. For example, must $q$ be divisible by an $(s+1)$-fold Pfister form in the situation of Corollary \ref{thmWittkernel}? Conjecturally, this question should have a positive answer in the case where $q$ has degree $s+1$, meaning here that $q$ does not represent an element in $I^{s+2}(F)$, the $(s+2)$-nd power of the fundamental ideal of even-dimensional forms in $W(F)$. We provide here some further evidence for this claim, including a proof for the case where $\mydim{p} \leq 16$ (Corollary \ref{corbinarymotive}). While it seems unlikely that such divisibility would hold in general, we are currently unable to provide any counterexample.

Theorem \ref{thmmain} is a consequence of results of Vishik (\cite{Vishik2}) on a certain descent problem for algebraic cycles over function fields of quadrics. The proof is given in \S \ref{secproof} below. Beyond this, we also prove Conjecture \ref{conjmain} in a number of additional cases; namely, the conjecture is true if (i) $2^{s+1}-2 \leq \mydim{p} \leq 2^{s+1}$, (ii) $p$ is a Pfister neighbour, (iii) $k \leq 7$, or (iv) $\mydim{q} \leq 2^{s+2} + 2^{s-1}$. In the case of (iii) and (iv), we again prove a stronger result at the level of the splitting pattern of $q$ (Theorem \ref{thmadditional}). In the course of treating case (iii), we show that (a stronger version of) our conjecture is implied by a long-standing conjecture of Kahn on the unramified Witt group of a quadric (\cite{Kahn1}, see \S \ref{subsecKahn} below).

We expect that the statement of Conjecture \ref{conjmain} is also true in characteristic $2$, even if we allow $p$ and $q$ to be degenerate.\footnote{Of course, if $q$ is degenerate, then one can no longer interpret the integer $k$ as the dimension of the anisotropic part of $q_{F(p)}$.} Unfortunately, the methods of the present article rely on the action of mod-2 Steenrod operations on Chow groups of smooth projective varieties, something which not available in characteristic $2$ at the present time. Nevertheless, using entirely different methods, we can show that our main results (and more) extend to the case where $\mathrm{char}(F) = 2$ and $q$ is \emph{quasilinear} (i.e., diagonalizable). In view of the different nature of the arguments used, this work has been confined to a separate text (\cite{Scully2}). 

We conclude this introduction with an example which shows that Conjecture \ref{conjmain} cannot be improved without stronger hypotheses or permitting the use of additional invariants:

\begin{example} \label{exoptimality} Let $p$ be an anisotropic quadratic form of dimension $\geq 2$ over a field $E$ of characteristic $\neq 2$. Suppose that $p$ is a neighbour of an $(s+1)$-fold Pfister form $\pi$ for some non-negative integer $s$. Choose a non-negative integer $a$, a non-negative integer $k<2^s$, and let $\epsilon = k - 2l$ for some integer $0 \leq l \leq k/2$. Note that we have $\epsilon + l \geq 0$. Let $X = (X_1,\hdots,X_{a + \epsilon + l})$ be a set of $a + \epsilon + l$ algebraically independent variables over $E$, and let $F = E(X)$. Let $\sigma$ be a codimension-$l$ subform of $\pi$, and consider the form
$$ q = \pi_F \otimes \form{X_1,\hdots,X_{a-1}} \perp X_a \sigma_F \perp \form{X_{a+1},\hdots,X_{a+\epsilon +l}}.$$
Since $F/E$ is a purely transcendental extension, $\pi_E$ and $\sigma_E$ are anisotropic. It then follows from \cite[Cor. 19.6]{EKM} that $q$ is anisotropic. Note that $q$ has dimension $a2^{s+1} + \epsilon$. We claim that $\anispart{(q_{F(p)})}$ has dimension $k$. Since $l \leq k/2 < 2^{s-1}$, $\sigma$ is a neighbour of $\pi$. Let $\tau$ be its complementary form of dimension $l$. Since $\pi$ is a Pfister form, we then have
$$ q_{F(p)} = \big(-X_a \tau \perp \form{X_{a+1},\hdots,X_{a+\epsilon+l}}\big)_{F(p)}$$
in $W(F(p))$. By Hoffmann's Separation theorem, $\tau_{E(p)}$ is anisotropic, and since $F(p)/E(p)$ is purely transcendental, \cite[Cor. 19.6]{EKM} then implies that the right-hand side of the above equation is anisotropic. We therefore have that
$$ \anispart{(q_{F(p)})} \simeq \big(-X_a \tau \perp \form{X_{a+1},\hdots,X_{a+\epsilon+l}}\big)_{F(p)}.$$
In particular, the dimension of $\anispart{(q_{F(p)})}$ is equal to
$$ \mydim{\tau} + \epsilon + l = 2l + \epsilon = k, $$
as claimed. Since the integers $k$ and $\epsilon$ in the statement of Conjecture \ref{conjmain} must have the same parity, this shows the optimality of the assertion. \end{example}

Before proceeding, we make the following conventions:\\

\noindent {\bf Conventions.} All fields considered in this paper have characteristic $\neq 2$, and all quadratic forms are finite-dimensional and non-degenerate. By a \emph{variety}, we mean an integral separated scheme of finite-type over a field. The field of $2$ elements is denoted by $\mathbb{F}_2$.

\section{Some preliminary facts and terminology} \label{secpreliminary} For the remainder of this text, $F$ will denote an arbitrary field of characteristic $\neq 2$. We assume basic familiarity with the algebraic theory of quadratic forms, and the reader is referred to \cite{EKM} for all undefined terminology and notation.

\subsection{The Knebusch splitting tower of a quadratic form} \label{subsecknebusch} Let $\varphi$ be a quadratic form over $F$. Recall the following construction of Knebusch (\cite{Knebusch1}): Set $F_0 = F$, $\varphi_0=\anispart{\varphi}$ (the anisotropic the kernel of $\varphi$), and recursively define
$$ F_r = F_{r-1}(\varphi_{r-1}) \hspace{1cm} \text{and} \hspace{1cm} \varphi_r = \anispart{(\varphi_{F_r})}, $$
with the understanding that the process stops at the first non-negative integer $h(\varphi)$ for which $\varphi_{h(\varphi)}$ is split (i.e., of dimension at most 1). The integer $h(\varphi)$ is called the \emph{height} of $\varphi$. By the \emph{splitting pattern} of $\varphi$, we will mean the decreasing sequence of integers comprised of the dimensions of the $\varphi_r$. For each $0 \leq r \leq h(\varphi)$, we set $\mathfrak{j}_r(\varphi)$ to be the Witt index of $\varphi$ after extension to the field $F_r$. If $\varphi$ is not split and $r \geq 1$, then the integer $\wittj{r}{\varphi} - \wittj{r-1}{\varphi}$ is called the \emph{r-th higher Witt index} of $\varphi$, and is denoted $\witti{r}{\varphi}$. By \cite[Thm. 5.8]{Knebusch1}, the even-dimensional anisotropic forms of height 1 are precisely those which are similar to a Pfister form, i.e., to a form of the shape $\pfister{a_1,\hdots,a_m} : = \bigotimes_{i=1}^m \form{1,-a_i}$. In view of the inductive nature of Knebusch's construction, it follows that if $\varphi$ is non-split of even dimension, then $\varphi_{h(\varphi)-1}$ is similar to an $n$-fold Pfister form for some positive integer $n$. The latter integer is called the \emph{degree} of $\varphi$, denoted $\mathrm{deg}(\varphi)$. If $\mathrm{dim}(\varphi)$ is odd (resp. if $\varphi$ is hyperbolic), then we set $\mathrm{deg}(\varphi) =0$ (resp. $\mathrm{deg}(\varphi) = \infty$). By a deep result due to Orlov, Vishik and Voevodsky (\cite[Thm. 4.3]{OVV}), $\mathrm{deg}(\varphi)$ then coincides with the supremum of the set of all integers $d$ for which $\varphi$ represents an element in the $d$-th power $I^d(F)$ of the fundamental ideal ideal in the Witt ring $W(F)$.

\subsection{Stable birational equivalence of quadratic forms} \label{subsecstablebirationality} Let $\psi$ and $\varphi$ be a pair of anisotropic quadratic forms over $F$ of dimension $\geq 2$. We say that $\varphi$ and $\psi$ are \emph{stably birationally equivalent} if both $\varphi_{F(\psi)}$ and $\psi_{F(\varphi)}$ are isotropic. For example, an anisotropic Pfister neighbour is stably birationally equivalent to its ambient Pfister form (\cite[Rem. 23.11]{EKM}). More generally, if $\psi$ is similar to a subform of $\varphi$ having codimension less than $\witti{1}{\varphi}$, then $\varphi$ and $\psi$ are stably birationally equivalent (see \cite[Lem. 74.1]{EKM}). Following the previous example, we say in this case that $\psi$ is a \emph{neighbour} of $\varphi$.

\subsection{Algebraic cycles on quadrics} \label{subsecChowsplit} Given a variety $X$ over a field $K$, we will write $\chow{X}$ for its total Chow group modulo 2. The group $\chow{X}$ has the natural structure of an $\mathbb{F}_2$-vector space. If $L$ is a field extension of $K$, then an element of $\chow{X_L}$ is said to be \emph{$K$-rational} if it lies in the image of the natural restriction homomorphism $\chow{X} \rightarrow \chow{X_L}$. Suppose now that $X$ is a smooth projective quadric of dimension $n \geq 1$ defined by the vanishing of a quadratic form $\varphi$ over our fixed field $F$, and let $\overline{F}$ be an algebraic closure of $F$. By \cite[Prop. 68.1]{EKM}, an $\mathbb{F}_2$-basis of $\chow{X_{\overline{F}}}$ is given by the set $\lbrace l_i, h^i\;|\; 0 \leq i \leq [n/2] \rbrace$, where $l_i$ (resp. $h^i$) is the class of an $i$-dimensional projective linear subspace (resp. a codimension-$i$ hyperplane section) of $X_{\overline{F}}$. The following lemma is a basic consequence of the well-know theorem of Springer asserting that odd-degree field extensions do not affect the Witt index of a quadratic form:

\begin{lemma}[{see \cite[Cor. 72.6]{EKM}}] \label{lemrationality} Let $\varphi$ be a quadratic form of dimension $\geq 2$ over $F$ with associated $($smooth$)$ projective quadric $X$, and let $0 \leq i \leq [n/2]$. If the element $l_i \in \chow{X_{\overline{F}}}$ is $F$-rational, then $\windex{\varphi} > i$. \end{lemma}

\subsection{The motivic decomposition type and upper motive of a quadric} \label{subsecmotives} Given a field $K$, we will write $\textsf{\textsl{Chow}}(K)$ for the additive category of Grothendieck-Chow motives over $K$ with \emph{integral} coefficients (as defined in \cite[Ch. XII]{EKM}, for example). If $X$ is a smooth projective variety over $K$, then we will write $M(X)$ for its motive considered as an object of $\textsf{\textsl{Chow}}(K)$. In the special case where $X = \mathrm{Spec}(K)$, we simply write $\mathbb{Z}$ instead of $M(X)$ (the dependency on the base field is suppressed from the notation). Given an integer $i$ and an object $M$ of $\textsf{\textsl{Chow}}(K)$, we will write $M\lbrace i \rbrace$ for the $i$-th Tate twist of $M$. In particular, $\mathbb{Z} \lbrace i \rbrace$ will denote the Tate motive with shift $i$ in $\textsf{\textsl{Chow}}(K)$. If $L$ is a field extension of $K$, we write $M_L$ to denote the image of an object $M$ in $\textsf{\textsl{Chow}}(K)$ under the natural scalar extension functor $\textsf{\textsl{Chow}}(K) \rightarrow \textsf{\textsl{Chow}}(L)$. 

Suppose now that $X$ is a smooth projective quadric of dimension $n \geq 1$ defined by the vanishing of a quadratic form $\varphi$ over our fixed field $F$, and let $\Lambda(n) = \lbrace i \;|\;0 \leq i \leq [n/2] \rbrace \coprod \lbrace n-i \;|\;0 \leq i \leq [n/2]\rbrace$. By a result of Vishik (see \cite[\S\S 3,4]{Vishik1}), any direct summand of $M(X)$ decomposes (in an essentially unique way) into a finite direct sum of indecomposable objects in $\textsf{\textsl{Chow}}(F)$. If $N$ is a non-zero direct summand of $M(X)$, then there exists a unique non-empty subset $\Lambda(N) \subseteq \Lambda(n)$ such that $N_{\overline{F}} \cong \bigoplus_{\lambda \in \Lambda(N)} \mathbb{Z}\lbrace \lambda \rbrace$ (\emph{loc. cit.}). In particular, we have $\Lambda\big(M(X)\big) = \Lambda(n)$, i.e., $M(X_{\overline{F}}) \simeq \bigoplus_{\lambda \in \Lambda(n)} \mathbb{Z}\lbrace \lambda \rbrace$. Now, if $N_1$ and $N_2$ are distinct \emph{indecomposable} direct summands of $M(X)$, then $\Lambda(N_1)$ and $\Lambda(N_2)$ are easily seen to be disjoint (\cite[Lem. 4.2]{Vishik1}), and so the complete motivic decomposition of $X$ determines in this way a partition of the set $\Lambda(n)$. This partition is an important discrete invariant of $\varphi$ known as its \emph{motivic decomposition type}. For the reader's convenience, we state here the most significant recent advance in the study of this invariant due to Vishik (\cite{Vishik4}). Recall first that the \emph{upper motive} of $X$ is defined as the unique indecomposable direct summand $U(X)$ of $M(X)$ such that $0 \in \Lambda(U(X))$, i.e., such that $\mathbb{Z}$ is isomorphic to a direct summand of $U(X)_{\overline{F}}$. We then have:

\begin{theorem}[{see \cite[Thm. 2.1]{Vishik4}}] \label{thmexcellentconnections} Let $\varphi$ be an anisotropic quadratic form of dimension $\geq 2$ over $F$ with associated $($smooth$)$ projective quadric $X$. Write
$$ \mathrm{dim}(\varphi) - \witti{1}{\varphi} = 2^{r_1} - 2^{r_2} + \cdots + (-1)^{t-1}2^{r_t} $$
for uniquely determined integers $r_1>r_2> \cdots >r_{t-1}>r_t + 1\geq 1$, and, for each $1 \leq c \leq t$, set
$$ D_c = \sum_{i=1}^{c-1} (-1)^{i-1}2^{r_i - 1} + \epsilon(c)\sum_{j=l}^t(-1)^{j-1}2^{r_j},$$
where $\epsilon(c) = 1$ if $c$ is even and $\epsilon(c) = 0$ if $c$ is odd. Then, for any $1 \leq c \leq t$, we have $D_c \in \Lambda(U(X))$, i.e., the Tate motive $\mathbb{Z}\lbrace D_c \rbrace$ is isomorphic to a direct summand of $U(X)_{\overline{F}}$.\end{theorem}

\section{Hyperbolicity of quadratic forms over function fields of quadrics}

The proof of Theorem \ref{thmmain} will be given in the next section. Taking this for granted, we give here the basic applications to the study of Witt kernels of function fields of quadrics. We start by recalling the statement of the Cassels-Pfister subform theorem:

\begin{proposition}[{see \cite[Lem. 4.5]{Knebusch1}}] \label{propCP} Let $p$ and $q$ be anisotropic quadratic forms of dimension $\geq 2$ over $F$ such that $q_{F(p)}$ is hyperbolic. Then, for any non-zero elements $a \in D(p)$ and $b \in D(q)$, there exists a quadratic form $r$ over $F$ such that $q \simeq r \perp abp$. In particular, $\mydim{p} \leq \mydim{q}$. \end{proposition}

Taking degrees into account, the dimension inequality can be refined as follows:

\begin{corollary} \label{corCP} Let $p$ and $q$ be anisotropic quadratic forms over $F$ such that $q_{F(p)}$ is hyperbolic, and let $n = \mathrm{deg}(q)$. Then $\mathrm{dim}(p) \leq 2^n$. \end{corollary}

Indeed, this follows from the proposition by taking $L$ to be the penultimate entry of the Knebusch splitting tower of $q$ in the following lemma:

\begin{lemma} \label{lemhypextensions} Let $q$ and $p$ be anisotropic quadratic forms of dimension $\geq 2$ over $F$ such that $q_{F(p)}$ is hyperbolic, and let $L$ be a field extension of $F$. If $q_L$ is not hyperbolic, then $p_L$ is anisotropic, and $\anispart{(q_L)}$ becomes hyperbolic over $L(p)$.
\begin{proof} It is clear that $\anispart{(q_L)}$ becomes hyperbolic over $L(p)$. If $p_L$ is isotropic, then $L(p)$ is a purely transcendental extension of $L$ (\cite[Prop. 22.9]{EKM}) and so $q_L$ must already be hyperbolic. \end{proof} \end{lemma}

\subsection{Main result} \label{subsecmainhyp} The statement of Corollary \ref{corCP} can be reformulated as follows: If $q_{F(p)}$ is hyperbolic, then the last higher Witt index of $q$ is equal to $2^m$ for some $m \geq s$. The new insight here is that the remaining higher Witt indices are all divisible by $2^{s+1}$:

\begin{theorem} \label{thmdivisibilityofindices} Let $p$ and $q$ be anisotropic quadratic forms of dimension $\geq 2$ over $F$, and let $s$ be the unique non-negative integer such that $2^s < \mydim{p} \leq 2^{s+1}$. If $q_{F(p)}$ is hyperbolic, then $\witti{r}{q}$ is divisible by $2^{s+1}$ for every $1 \leq r < h(q)$. 

\begin{proof} Since the statement is insensitive to multiplying $p$ and $q$ by non-zero scalars, we can assume that both forms represent $1$. We argue by induction on $h(q)$. If $h(q)=1$, there is nothing to prove, so assume otherwise. Then, by Lemma \ref{lemhypextensions}, $p_{F(q)}$ is anisotropic and $q_1$ becomes hyperbolic over $F(q)(p)$. The induction hypothesis therefore implies that $\witti{r}{q} = \witti{r-1}{q_1}$ is divisible by $2^{s+1}$ for all $1<r<h(q)$. It remains to show that $\witti{1}{q}$ is also divisible by $2^{s+1}$. We begin by  showing that $\witti{1}{q} > 2^s$.  Suppose, for the sake of contradiction, that this is not the case. Since both $p$ and $q$ represent $1$, it follows from the Cassels-Pfister subform theorem (Proposition \ref{propCP} above) that there exists a quadratic form $r$ over $F$ such that
\begin{equation} \label{eq3.1} q \simeq r \perp p. \end{equation}
We recall the following observation due to  Gentile and Shapiro:

\begin{lemma}[{\cite[Lem. 17]{GentileShapiro}}] \label{lemGS} Let $L$ be a field extension of $F$. In the above situation, either
\begin{enumerate} \item $\anispart{(q_L)} \simeq \anispart{(r_L)} \perp \anispart{(p_L)}$, or
\item $\anispart{(r_L)} \simeq \anispart{(q_L)}  \perp - \anispart{(p_L)}$. 
\end{enumerate}
\begin{proof} Both (1) and (2) are true if we replace the isometry relation with Witt equivalence. In particular, (2) holds if either $q_L$ is hyperbolic or $\anispart{(q_L)} \perp - \anispart{(p_L)}$ is anisotropic. If not, then $p_L$ is anisotropic and $\anispart{(q_L)}$ becomes hyperbolic over $L(p)$ (Lemma \ref{lemhypextensions}). At the same time, $\anispart{(q_L)}$ and $p_L$ represent a common non-zero element of $L$, and so $p_L \subset \anispart{(q_L)}$ by the Cassels-Pfister subform theorem. By Witt cancellation, it then follows that (1) holds. \end{proof}\end{lemma}

We apply this lemma in the case where $L = F(q)$. Note first that since $\mydim{p} > 2^s$, $r$ has codimension $>2^s$ in $q$. Since $\witti{1}{q} \leq 2^s$, it follows from \cite[Cor. 4.9]{Vishik1} that $r_{F(q)}$ is anisotropic. As $p_{F(q)}$ is also anisotropic, the lemma therefore tells us that either
\begin{enumerate} \item $q_1 \simeq r_{F(q)} \perp p_{F(q)}$, or
\item $r_{F(q)} \simeq q_1 \perp - p_{F(q)}$. \end{enumerate}
Now (1) cannot hold, since the right-hand side has dimension equal to $\mydim{q}$, which is larger than $\mydim{q_1}$. On the other hand, (2) cannot hold either; indeed, if (2) holds, then
\begin{eqnarray*} \witti{1}{q} &=& \frac{\mydim{q} - \mydim{q_1}}{2} \\
    &=& \frac{\big(\mydim{r} + \mydim{p}\big) - \big(\mydim{r} - \mydim{p}\big)}{2}\\
    &=& \mydim{p}> 2^s, \end{eqnarray*}
which contradicts our assumption. We therefore conclude that $\witti{1}{q} > 2^s$. By Karpenko's theorem on the possible values of the first higher Witt index (\cite{Karpenko2}), it follows that $\mathrm{dim}(q) - \witti{1}{q}$ is divisible by $2^{s+1}$. Now the $k=0$ case of Theorem \ref{thmmain} tells us that $\mydim{q}$ is divisible by $2^{s+1}$, and so the same is therefore true of $\witti{1}{q}$. \end{proof} \end{theorem}

In particular, we get the following more precise version of Corollary \ref{thmWittkernel}:

\begin{corollary} \label{corWittkernel} Let $q$ and $p$ be anisotropic quadratic forms of dimension $\geq 2$ over $F$, let $s$ be the unique non-negative integer such that $2^s < \mathrm{dim}(p) \leq 2^{s+1}$, and let $n = \mathrm{deg}(q)$. If $q_{F(p)}$ is hyperbolic, then $n \geq s+1$, and:
\begin{enumerate} \item If $n = s+1$, then $\mathrm{dim}(q) = 2^{s+1}m$ for some odd integer $m$.
\item If $n \geq s+2$, then $\mathrm{dim}(q)$ is divisible by $2^{s+2}$. \end{enumerate}
\begin{remark} We remind the reader that, by \cite[Thm. 4.3]{OVV}, $n = \mathrm{deg}(q)$ coincides with the largest integer $d$ such that $[q] \in I^d(F)$. \end{remark} 
\begin{proof} The inequality $n \geq s+1$ holds by Corollary \ref{corCP}. In view of the fact that
\begin{equation} \label{eq3.2} \mathrm{dim}(q) = 2^n + \sum_{r=1}^{h(q)-1} 2\witti{r}{q}, \end{equation}
the remaining statements follow immediately from Theorem \ref{thmdivisibilityofindices}. \end{proof} \end{corollary}

Now, the following statement is well known (see \cite[Lem. 6.2]{Vishik1} for a proof):

\begin{proposition} Let $q$ be an anisotropic quadratic form over $F$ which is divisible by an $m$-fold Pfister form for some integer $m \geq 1$. Then $\witti{r}{q}$ is divisible by $2^{m}$ for all $1 \leq r < h(q)$. \end{proposition}

Theorem \ref{thmdivisibilityofindices} therefore raises the following question:

\begin{question} Let $p$ and $q$ be anisotropic quadratic forms of dimension $\geq 2$ over $F$, and let $s$ be the unique non-negative integer such that $2^s < \mydim{p} \leq 2^{s+1}$. If $q_{F(p)}$ is hyperbolic, does it follow that $q$ is divisible by an $(s+1)$-fold Pfister form? \end{question}

While it perhaps seems unlikely that this would hold in general, we are unable to provide a counterexample. At the same time, we do expect that the question has a positive answer in the case where $\mathrm{deg}(q)$ takes its smallest possible value of $s+1$. In fact, the hyperbolicity of $q_{F(p)}$ should imply that $p$ is a Pfister neighbour in this case; by \cite{ArasonPfister}, this would give that $q$ is isometric to the product of an odd-dimensional form and the ambient Pfister form of $p$. This expectation goes back to Fitzgerald (\cite{Fitzgerald2}), and we provide further evidence in its favour in \S \ref{secadditionalresults} below. In particular, we show that if $\mathrm{deg}(q) = s+1$ and $q_{F(p)}$ is hyperbolic, then the upper motive of the quadric $\lbrace p = 0 \rbrace$ is a twisted form of a Rost motive (Corollary \ref{corbinarymotive}). By a result of Karpenko (\cite{Karpenko1}), this implies that our claim holds if $s \leq 3$. Hypothetically, it should imply for all $s$ (\cite[Conj. 4.21]{Vishik1}).

\subsection{Fitzgerald's theorem revisited} \label{subsecfitz} The proof of Corollary \ref{corWittkernel} immediately yields the following lower bound for the dimension of an element of the Witt kernel $W(F(p)/F)$:

\begin{corollary} \label{corWittkernellowerbound} Let $p$ and $q$ be anisotropic quadratic forms of dimension $\geq 2$ over $F$, let $s$ be the unique non-negative integer such that $2^s<\mydim{p} \leq 2^{s+1}$, and let $n = \mathrm{deg}(q)$. If $q_{F(p)}$ is hyperbolic, then $n \geq s+1$ and 
$$ \mathrm{dim}(q) \geq 2^n + (h(q)-1)2^{s+2}, $$
with equality holding if and only if $\witti{r}{q} = 2^{s+1}$ for all $1 \leq r < h(q)$. \end{corollary}

In particular, we get a satisfying explanation of Fitzgerald's theorem on  the low-dimensional part of $W(F(p)/F)$:

\begin{corollary}[{Fitzgerald, \cite[Thm. 1.6]{Fitzgerald1}}] \label{corFitzgerald} Let $p$ and $q$ be anisotropic quadratic forms of dimension $\geq 2$ over $F$ such that $q_{F(p)}$ is hyperbolic, and let $n = \mathrm{deg}(q)$. If $\mathrm{dim}(p) > \frac{1}{2}(\mathrm{dim}(q) - 2^n)$, then $q$ is similar to a Pfister form.
\begin{proof} To say that $q$ is similar to a Pfister form is equivalent to saying that $h(q) =1$ (\cite[Thm. 5.8]{Knebusch1}). Let $s$ be the unique non-negative integer such that $2^s < \mathrm{dim}(p) \leq 2^{s+1}$. Then, by hypothesis, we have
$$ \mathrm{dim}(q) < 2^n + 2\mathrm{dim}(p) \leq 2^n + 2^{s+2}. $$
By Corollary \ref{corWittkernellowerbound}, it follows that $h(q) = 1$, as desired.\end{proof}\end{corollary}

\begin{remark} As noted by Fitzgerald (see \cite[\S 2]{Fitzgerald1}), the inequality of Corollary \ref{corFitzgerald} cannot be weakened in general. Here we can clarify the situation further: Let $p$ and $q$ be as in the statement of Corollary \ref{corWittkernellowerbound}. If we are in the border case where $\mathrm{dim}(p) = \frac{1}{2}(\mathrm{dim}(q) - 2^n)$, then our result shows that $h(q) = 2$, $\mathrm{dim}(q) = 2^n + 2^{s+2}$ and $\mathrm{dim}(p) = 2^{s+1}$. By a theorem of Vishik (\cite[Thm. 6.4]{Vishik1}), any degree-$n$ form of height $\geq 2$ has dimension $\geq 2^n + 2^{n-1}$. We therefore have $n-3 \leq s \leq n-1$. All three cases can occur. In fact, $q$ should, in this situation, be the product of an $(s+1)$-fold Pfister form and either an Albert form (if $s=n-3$), a $4$-dimensional form of non-trivial discriminant (if $s=n-2$) or a form of dimension $3$ (if $s = n-1$) -- see \cite[\S 2.4]{Kahn2}. This has been shown to be the case when $n \leq 3$ (\emph{loc. cit.}), and can also be proven for $n=4$ using more recent results (though this is absent from the literature). The general case remains out of reach at present. \end{remark}

\section{Proof of Theorem \ref{thmmain}} \label{secproof} In this section, we give a concrete geometric reformulation of Conjecture \ref{conjmain}. Theorem \ref{thmmain} then follows as a consequence of the results of \cite{Vishik2}. We begin by making the following observation:

\begin{lemma} \label{lemreformulatedconj} Let $d$, $k$ and $s$ be positive integers. Assume that $d+k$ is even and that $k < 2^s$. Then the following are equivalent:
\begin{enumerate} \item $d= a2^{s+1} + \epsilon$ for some non-negative integer $a$ and some integer $-k \leq \epsilon \leq k$.
\item $\frac{d+k}{2} = a2^s + \mu$ for some non-negative integer $a$ and some integer $0 \leq \mu \leq k$.
\item The binomial coefficient
$$ \binom{\frac{d+k}{2}}{l} $$
is even for every $k < l \leq m$, where $m = \begin{cases} 2^{s} -2 & \text{if $k \geq 2^{s-1}$} \\
                                                        2^{s-1}  & \text{if $k < 2^{s-1}.$} \end{cases}$ \end{enumerate}
\begin{proof} The equivalence of (1) and (2) is clear.
To prove the equivalence of (2) and (3), we need to recall the basic criterion for oddness of binomial coefficients. Let $x$ be a non-negative integer, and let $\sum_{i=0}^\infty x_i2^i$ be its $2$-adic expansion. We write $\pi(x)$ for the (finite) set of all non-negative integers $i$ for which $x_i$ is non-zero. We then have:

\begin{lemma}[{see \cite[Lem. 3.4.2]{Springer}}] \label{lembinomialparity} Let $x$ and $y$ be non-negative integers. Then the binomial coefficient $\binom{x}{y}$ is odd if and only if $\pi(y) \subset \pi(x)$. \end{lemma}

Returning to the proof of Lemma \ref{lemreformulatedconj}, let us write
$$ \frac{d+k}{2} = a2^s + \mu $$
for some non-negative integer $a$ and some integer $0 \leq \mu < 2^s$. If $0 \leq l < 2^s$, then Lemma \ref{lembinomialparity} implies that
$$ \binom{\frac{d+k}{2}}{l} \equiv \binom{\mu}{l} \pmod{2}. $$
The equivalence $(2) \Leftrightarrow (3)$ therefore amounts to the assertion that $\mu \leq k$ if and only if $\binom{\mu}{l}$ is even for every $k < l \leq m$. The ``only if'' implication here is trivial. Conversely, suppose that $\binom{\mu}{l}$ is even for every $k < l \leq m$. Since $\binom{\mu}{\mu} =1$, it follows that either $\mu \leq k$ or $\mu >m$. If $\mu > m$, then (by the definition of $m$) we either have
\begin{enumerate} \item[(i)] $\mu = 2^s-1$, or
\item[(ii)] $k < 2^{s-1} = m < \mu$. \end{enumerate}
By Lemma \ref{lembinomialparity}, $\binom{2^s-1}{l}$ is odd for every $0\leq l \leq 2^{s-1}$, so if (i) holds, then we must have that $k = 2^s-1=m$. If (ii) holds, then $\binom{\mu}{2^{s-1}}$ is even by hypothesis. On the other hand, the inequalities $2^{s-1} < \mu < 2^s$ imply in this case that $s-1 \in \pi(\mu)$. By Lemma \ref{lembinomialparity}, this means that $\binom{\mu}{2^{s-1}}$ is odd, giving us a contradiction. We conclude that $\mu \leq k$, and so the equivalence of (2) and (3) is proved. \end{proof} \end{lemma}

Lemma \ref{lemreformulatedconj} reduces Conjecture \ref{conjmain} to the problem of checking the parity of certain binomial coefficients. The point now is that these coefficients can be extracted from the action of the Steenrod algebra on the mod-2 Chow ring of the quadric defined by $q$. For the remainder of this section, let $p$ and $q$ be anisotropic quadratic forms of dimension $\geq 2$ over $F$, and let $Q$ be the projective quadric of equation $q=0$. We fix an algebraic closure $\overline{F}$ (resp. $\overline{F(p)}$) of $F$ (resp. $F(p)$). For any $j \geq 0$, let $S^j$ denote the $j$-th Steenrod operation of cohomological type on the mod-2 Chow ring of a smooth quasi-projective variety.

\begin{lemma}\label{lembinomialeven} Fix integers $0 \leq r \leq [\mydim{Q}/2]$ and $0 \leq j \leq r$. Then the element $S^j(l_r) \in \chow{Q_{\overline{F}}}$ is $F$-rational if and only if the binomial coefficient
$$ \binom{\mydim{q} -r - 1}{\mydim{q} - r - 1 - j}$$
is even.
\begin{proof} By \cite[Cor. 78.5]{EKM}, we have
$$ S^j(l_r) = \binom{\mydim{q} -r - 1}{j} l_{r-j}. $$
If $S^j(l_r)$ is $F$-rational, then the coefficient here must be even by Lemma \ref{lemrationality} (because $q$ is anisotropic). The converse is trivial. \end{proof} \end{lemma}

If $Y$ is a variety over $F$, then an element $\alpha \in \chow{Y_{\overline{F}}}$ will be called \emph{$F(p)$-rational} if $\alpha_{\overline{F(p)}} \in \chow{Y_{\overline{F(p)}}}$ is $F(p)$-rational. We recall the following result of Vishik (extended to odd characteristic in \cite{Fino}):

\begin{theorem}[{\cite[Cor. 3.5]{Vishik2}, \cite[Thm. 1.1, Prop. 2.1]{Fino}}] \label{thmmaintool} Let $Y$ be a smooth quasi-projective variety over $F$, and let $\alpha$ be an $F(p)$-rational element of $\mathrm{Ch}^m(Y_{\overline{F}})$. If the integral Chow group $\mathrm{CH}(Y_{\overline{F}})$ is torsion-free, then:
\begin{enumerate} \item $S^j(\alpha)$ is $F$-rational for every $j > m - \left[(\mydim{p}-1)/2\right]$.
\item If $p_{F(Y)}$ is not split, then $S^j(\alpha)$ is also $F$-rational for $j = m - \left[(\mydim{p}-1)/2\right]$. \end{enumerate} \end{theorem}

Applying Theorem \ref{thmmaintool} to the quadric $Q$, we get the following:

\begin{corollary} \label{corvishik} Let $r =  \windex{q_{F(p)}}-1$, and consider the element $l_r \in \chow{Q_{\overline{F}}}$. Then:
\begin{enumerate} \item $S^j(l_r)$ is $F$-rational for all $j > \mydim{Q}-r - [(\mydim{p}-1)/2]$.
\item If $p_{F(q)}$ is not split, then $S^j(l_r)$ is also $F$-rational for $j = \mydim{Q} -r - [(\mydim{p}-1)/2]$ \end{enumerate}
\begin{proof} Since $r = \windex{q_{F(p)}} - 1$, the element $l_r$ is $F(p)$-rational. As it is represented by a cycle of dimension $r$, the assertions follow immediately from Vishik's theorem. \end{proof} \end{corollary}

Now, as in the statement of Conjecture \ref{conjmain}, let $k = \mydim{q} - 2\windex{q_{F(p)}}$ (i.e., let $k$ be the dimension of the anisotropic part of $q_{F(p)}$). 

\begin{corollary} \label{corbinomials} The binomial coefficient
$$ \binom{\frac{\mydim{q} +k}{2}}{l} $$
is even for all $k<l \leq [(\mydim{p}-1)/2]$ $($and also for $l = [(\mydim{p}-1)/2]+1$ if $p_{F(q)}$ is not split$)$.
\begin{proof} To simplify the notation, let $n = [(\mydim{p}-1)/2]$. Let $r = \windex{q_{F(p)}}-1$, and let $\mydim{Q} - r-n \leq j \leq r$. If $j = \mydim{Q} - r-n$, assume additionally that $p_{F(q)}$ is not split. Then, by Corollary \ref{corvishik}, $S^j(l_r)$ is $F$-rational. Since $j \leq r$, Lemma \ref{lembinomialeven} then implies that the binomial coefficient
$$ \binom{\mydim{q} - r - 1}{\mydim{q} - r - 1 - j}$$
is even. It now only remains to observe that 
$$ \mydim{q} - r -1 = \frac{\mydim{q}+k}{2},$$
and that
$$ \mydim{Q} - r-n \leq j \leq r $$
if and only if
$$ k < \mydim{q} - r - 1 - j \leq n+1.  $$ \end{proof} \end{corollary}

This yields Theorem \ref{thmmain}, as well as some other cases of Conjecture \ref{conjmain}:

\begin{theorem} Conjecture \ref{conjmain} is true in the following cases:
\begin{enumerate} \item The case where $k<2^{s-1}$.
\item The case where $2^{s+1} - 2 \leq \mydim{p} \leq 2^{s+1}$.
\item The case where $p$ is a Pfister neighbour. \end{enumerate}
\begin{proof} Let $p$, $q$, $s$ and $k$ be as in the statement of Conjecture \ref{conjmain}, and let 
$$ m=\begin{cases} 2^{s} -2 & \text{if $k \geq 2^{s-1}$} \\
                                                        2^{s-1}  & \text{if $k < 2^{s-1}$}. \end{cases}$$
By Lemma \ref{lemreformulatedconj}, the statement of the conjecture holds for the pair $(p,q)$ if and only if
$$\binom{\frac{\mydim{q}+k}{2}}{l}$$
is even for every $k < l \leq m$. Corollary \ref{corvishik} shows that this holds in both cases (1) and (2). Finally, suppose that $p$ is a Pfister neighbour with ambient Pfister form $\pi$. Since $p$ and $\pi$ are stably birationally equivalent (see \S \ref{subsecstablebirationality}), we have $\windex{q_{F(\pi)}} = \windex{q_{F(p)}}$. To prove that the statement of the conjecture holds in this case, we can therefore replace $p$ by $\pi$ and assume that $p$ is a Pfister form. In particular, we can assume that $\mydim{p} = 2^{s+1}$. The validity of the conjecture in this case therefore follows from its validity in case (2). \end{proof} \end{theorem}

Further partial results towards Conjecture \ref{conjmain} will be given in the next section, but the case where $k\geq 2^{s-1}$ remains open in general. From the above arguments, it is easy to see that we have the following reformulation of the conjecture as a refinement of \cite[Thm. 3.1]{Vishik2} for quadrics. The details are left to the reader:

\begin{proposition} The following are equivalent:
\begin{enumerate} \item Conjecture \ref{conjmain} is true.
\item Let $p$ be an anisotropic quadratic form of dimension $\geq 2$ over $F$, and let $s$ be the unique non-negative integer such that $2^s < \mydim{p} \leq 2^{s+1}$. Let $Y$ be a $($possibly isotropic$)$ smooth projective $F$-quadric. If $\alpha \in \mathrm{Ch}^m(Y_{\overline{F}})$ is $F(p)$-rational, then $S^j(\alpha)$ is $F$-rational for all $j \geq m - (2^s-2)$.  \end{enumerate} \end{proposition}

\section{Additional results} \label{secadditionalresults} In this section, we provide further justification for Conjecture \ref{conjmain} beyond Theorem \ref{thmmain}. First, we prove that the conjecture holds under a certain assumption on the splitting pattern of $q$. We then explain how the remaining cases would follow if a related, but rather stronger conjecture of Kahn on the unramified Witt ring of a quadric were true. This approach reveals that, in analogy with Theorem \ref{thmdivisibilityofindices}, Conjecture \ref{conjmain} should extend to a statement at the level of the splitting pattern of $q$. We begin by stating this extension.

\subsection{A refinement of the main conjecture} Let $p$ and $q$ be anisotropic quadratic forms of dimension $\geq 2$ over $F$, and let $k = \mydim{q} - 2\windex{q_{F(p)}}$ (i.e., $k$ is the dimension of the anisotropic part of $q_{F(p)}$). By \cite[\S 5]{Knebusch1}, $k$ is equal to the dimension of $q$ or one of its higher anisotropic kernels. We can therefore propose the following:

\begin{conjecture} \label{conjrefined} In the above situation, let $0 \leq l \leq h(q)$ be such that $\mydim{q_l} = k$. Suppose that $k<2^s$, where $s$ is the unique non-negative integer such that $2^s < \mydim{p} \leq 2^{s+1}$. Then, for each $0 \leq r < l$, there exist integers $a_r,b_r \geq 1$ and $-k \leq \epsilon_r \leq k$ such that
\begin{enumerate} \item $\mydim{q_r} = 2^{s+1}a_r + \epsilon_r$.
\item With one possible exception, either $\witti{r+1}{q} \leq \frac{1}{2}(k + \epsilon_r)$ or $\witti{r+1}{q} = b_r2^{s+1} + \epsilon_r$. The exception is where $\mydim{q_r} = 2^{s+1} - k$, in which case $r=l-1$ and $\witti{r+1}{q} = 2^s - k$.
\end{enumerate} \end{conjecture}

\begin{remark} If $k=0$ (i.e., if $q_{F(p)}$ is hyperbolic), then $l = h(q) - 1$, and this is just Theorem \ref{thmdivisibilityofindices}. In general, the exceptional case of (2) corresponds to the case of Theorem \ref{thmdivisibilityofindices} in which $\mathrm{deg}(q) = s+1$, so that $\witti{h(q)}{q}$ is not divisible by $2^{s+1}$ (being equal to $2^s$). \end{remark} 

Observe that, for $k<2^s$, Conjecture \ref{conjmain} correponds to assertion (1) of Conjecture \ref{conjrefined} in the case where $r=0$ (of course, Conjecture \ref{conjmain} is vacuously true if $k \geq 2^s$). We will prove here the following:

\begin{theorem} \label{thmadditional} Conjecture \ref{conjrefined} $($and hence Conjecture \ref{conjmain}$)$ is true if either
\begin{enumerate}\item $ k \leq 7$. 
\item $\mydim{q} \leq 2^{s+2} + k$. \end{enumerate} \end{theorem}

The first case, in particular, implies that both our conjectures hold if $\mydim{p} \leq 16$. The proof of Theorem \ref{thmadditional} is given in \S \ref{subsecadditional} below. Case (2) is treated using a motivic argument which in fact gives more (Theorem \ref{thmSPcondition}). Case (1), on the other hand, is treated by relating Conjecture \ref{conjrefined} to Theorem \ref{thmdivisibilityofindices} and the aforementioned conjecture of Kahn. We we now explain, this approach more clearly illustrates the philosophy underlying this paper.

\subsection{A conjecture of Kahn} \label{subsecKahn} Let $\varphi$ be an anisotropic quadratic form of dimension $\geq 2$ over $F$. Recall that the \emph{unramified Witt ring} of the projective quadric $\lbrace \varphi = 0 \rbrace$, denoted $W_{\mathrm{nr}}(F(\varphi))$, is defined as the subring of $W(F(\varphi))$ consisting of those classes of anisotropic quadratic forms which have no non-trivial (second) residues at the codimension-1 points of the projective quadric $\lbrace \varphi=0 \rbrace$ (see \cite{Kahn1}). The image of the canonical scalar extension map $W(F) \rightarrow W(F(\varphi))$ trivially lies in $W_{\mathrm{nr}}(F(\varphi))$, but equality need not hold in general (see, e.g., \cite[Cor. 10]{KahnRostSujatha}). Nevertheless, Kahn has made the following strong conjecture concerning the ``low-dimensional'' part of $W_{\mathrm{nr}}(F(\varphi))$:

\begin{conjecture}[{Kahn, \cite[Conj. 1]{Kahn1}}] \label{conjKahn} Let $\varphi$ be an anisotropic quadratic form over $F$ and let $[\tau] \in W_{nr}(F(\varphi))$. If $\mathrm{dim}(\tau) < \frac{1}{2} \mathrm{dim}(\varphi)$, then $\tau$ is defined over $F$, i.e., there exists a quadratic form $\sigma$ over $F$ such that $\tau \simeq \sigma_{F(\varphi)}$. \end{conjecture}

If true, this has the following application to Conjecture \ref{conjrefined}:

\begin{proposition} \label{propdefined} Assume in the situation of Conjecture \ref{conjrefined} that $\mathrm{dim}(q_{l-1}) > 2k$. Then, $\anispart{(q_{F(p)})}$ is defined over $F$ provided that Conjecture \ref{conjKahn} holds whenever $\mathrm{dim}(\tau) \leq k$ and $\mathrm{dim}(\varphi) \geq \mathrm{dim}(q_{l-1})$.
\begin{proof} Let $(F_r)$ be the Knebusch splitting tower of $q$. The form $q_l$ is evidently unramified at the codimension-1 points of the quadric $\lbrace q_{l-1} = 0 \rbrace$. Since $\mathrm{dim}(q_l) = k$ and $\mathrm{dim}(q_{l-1}) > 2k$, our hypothesis implies that $q_l \simeq \tau_{F_l}$ for some form $\tau$ over $F_{l-1}$. If $l>1$, then $\tau$ is unramified at the codimension-1 points of $X = \lbrace q_{l-2} = 0 \rbrace$. Indeed, since $q_l \simeq \tau_{F_l}$, Corollary \ref{corCP} implies that $[\tau \perp -q_{l-1}] \in I^n(F_{l-1})$, where $2^n \geq \mathrm{dim}(q_{l-1}) > 2k$. Thus, if $x \in X^{(1)}$, and $\pi$ is a uniformizer for $\mathcal{O}_{X,x}$, we have (using \cite[Lem. 19.14]{EKM})
$$ \partial_{x,\pi}([\tau]) = \partial_{x,\pi}([\tau \perp - q_{l-1}]) \in \partial_{x,\pi}\big(I^n(F_{l-1})\big) \subseteq I^{n-1}(F_{l-2}(x)).$$
But $\mathrm{dim}(\tau) = k$, and since $2^{n-1}>k$, the Arason-Pfister Hauptsatz (\cite{ArasonPfister}) then implies that $\partial_{x,\pi}([\tau]) = 0$. Again, our hypothesis now gives that $q_l$ is defined over $F_{l-2}$. Continuing in this way, we see that $q_l$ is defined over $F$, say $q_l \simeq \sigma_{F_l}$. We claim that $\anispart{(q_{F(p)})} \simeq \sigma_{F(p)}$. Since $\mathrm{dim}(\sigma) = k = \mathrm{dim}(\anispart{(q_{F(p)})})$, it suffices to show that $(q \perp - \sigma)_{F(p)}$ is hyperbolic. But $(q \perp - \sigma)_{F_l(p)}$ is hyperbolic, and so, by \cite[Lem. 7.15]{EKM}, it only remains to observe that $F_l(p)$ is a purely transcendental extension of $F(p)$. But if $r<l$, then the form $q_r$ becomes isotropic over $F_r(p)$ by the very definition of $k$. Hence $F_{r+1}(p) = F_r(p)(q_r)$ is a purely transcendental extension of $F_r(p)$, and so the claim follows by an easy induction. \end{proof} \end{proposition}

If one assumes another related conjecture of Vishik, then one can go a bit further here (see Remark \ref{remsfinal} (2)). In any case, the philosophy underlying Conjecture \ref{conjrefined} is now clear: If $q_{F(p)}$ is ``nearly hyperbolic'', then its anisotropic part should typically be defined over $F$, and the necessary conclusions can be made using Theorem \ref{thmdivisibilityofindices}. In more details:

\begin{lemma} \label{lemdefinedimpliesconj} Conjecture \ref{conjrefined} is true in the case where $\anispart{(q_{F(p)})}$ is defined over $F$. 
\begin{proof} By an easy induction on $l$, it suffices to show that if $l>0$, then $\mathrm{dim}(q)$ and $\witti{1}{q}$ have the prescribed form. Assume that $\anispart{(q_{F(p)})} \simeq \sigma_{F(p)}$ for some form $\sigma$ over $F$, and let $\eta = \anispart{(q \perp - \sigma)}$. Since both $q$ and $\sigma$ are anisotropic, there exists an integer $0 \leq \lambda \leq \mathrm{dim}(\sigma) = k$, and codimension-$\lambda$ subforms $\widetilde{q} \subset q$ and $\widetilde{\sigma} \subset \sigma$ such that $\eta \simeq \widetilde{q} \perp -\widetilde{\sigma}$. In particular, setting $\epsilon = 2\lambda - k$, we have $\mathrm{dim}(q) = \mathrm{dim}(\eta) + \epsilon$. Now, by construction, $\eta$ becomes hyperbolic over $F(p)$. Thus, by Corollary \ref{corWittkernel}, we have $\mathrm{dim}(\eta) = 2^{s+1}a$ for some $a \geq 1$ (and also $\mathrm{deg}(\eta) \geq s+1$). We therefore have $\mathrm{dim}(q) = 2^{s+1}a + \epsilon$. As for the assertion regarding $\witti{1}{q}$, we may assume that $\witti{1}{q} > (k+\epsilon)/2 = \lambda$, so that $\widetilde{q}$ is a neighbour of $q$ in the sense of \S \ref{subsecstablebirationality} above. This implies, in particular, that $q$ and $\widetilde{q}$ are stably birationally equivalent. If $\eta$ is similar to an $(s+1)$-fold Pfister form, then the Cassels-Pfister subform theorem implies that $q$ is a subform of $\eta$, whence $\eta \simeq q \perp -\sigma$ by Witt cancellation. Since $\mathrm{dim}(\sigma) = k < 2^s$, $q$ is then a neighbour of $\eta$, and so $\witti{1}{q} = 2^s - k$ by \cite[Cor. 4.9]{Vishik1}. Otherwise, Theorem \ref{thmdivisibilityofindices} implies that $\witti{1}{\eta} = 2^{s+1}b$ for some $b \geq 1$. Again, since $\mathrm{dim}(\widetilde{\sigma}) < 2^s$, it follows that $\widetilde{q}$ is a neighbour of $\eta$. As $q$ is a neighbour of $\widetilde{q}$, we see that $q$ and $\eta$ are stably birationally equivalent, and so $\witti{1}{q} = \witti{1}{\eta} + \mathrm{dim}(q) - \mathrm{dim}(\eta) = 2^{s+1}b + \epsilon$ by another application of \cite[Cor. 4.9]{Vishik1}. This proves the lemma. \end{proof} \end{lemma}

\subsection{Proof of Theorem \ref{thmadditional}} \label{subsecadditional} For the remainder of the paper, we fix the following notation:

\begin{itemize} \item $q$ and $p$ are anisotropic quadratic forms of dimension $\geq 2$ over $F$.
\item $Q$ and $P$ denote the (smooth) projective quadrics defined by the vanishing of $q$ and $p$, respectively.
\item $s$ is the unique non-negative integer such that $2^s < \mathrm{dim}(p) \leq 2^{s+1}$.
\item $k := \mydim{q} - 2\windex{q_{F(p)}}$ (i.e., $k$ is the dimension of the anisotropic part of $q_{F(p)}$). 
\item $l$ is the unique integer in $[0,h(q)]$ such that $\mathrm{dim}(q_l) = k$. 
\item $\mathfrak{i} := \windex{q_{F(p)}}$. \end{itemize}

We now proceed with the proof of Theorem \ref{thmadditional}. We begin by showing that Conjecture \ref{conjrefined} is true provided that $2^{s+1} - k$ lies in the splitting pattern of $q$. To make this more precise, we note the following:

\begin{lemma} \label{lemdimatprevious} Assume, in the situation of Conjecture \ref{conjrefined}, that $l\geq 1$ $($i.e, that $q_{F(p)}$ is isotropic$)$. Then $\mathrm{dim}(q_{l-1}) = 2^N - k$ for some integer $N \geq s+1$. 
\begin{proof} Let $(F_r)$ be the Knebusch splitting pattern of $q$. Since $q_{l-1}$ becomes isotropic over $F_{l-1}(p)$, and since $\mathrm{dim}(p) > 2^s$, Hoffmann's Separation theorem (\cite[Thm. 1]{Hoffmann1}) implies that $\mathrm{dim}(q_{l-1}) > 2^s$. In particular, $\mathrm{dim}(q_{l-1}) = 2^N - m$ for some integers $N \geq s+1$ and $0 \leq m < 2^{N-1}$. By definition, we then have $\witti{1}{q_{l-1}} = \frac{1}{2}(2^N - m -k)$, so that
$$ \mathrm{dim}(q_{l-1}) - \witti{1}{q_{l-1}} = \frac{1}{2}(2^N - m + k) > 2^{N-2} \hspace{.5cm} (\text{remember that } k<2^s). $$
By Karpenko's theorem on the possible values of the first Witt index (\cite{Karpenko2}), it follows that $2^{N-1}$ divides $m-k$. Since both both $m$ and $k$ are strictly less than $2^{N-1}$, this implies that $m=k$, and so $\mathrm{dim}(q_{l-1}) = 2^N - k$.
\end{proof} \end{lemma}

Our result is now that Conjecture \ref{conjrefined} holds in the case where $N = s+1$:

\begin{theorem} \label{thmSPcondition} Assume that $l \geq 1$ $($i.e., that $q_{F(p)}$ is isotropic$)$. If $\mathrm{dim}(q_{l-1}) = 2^{s+1} - k$, then Conjecture \ref{conjrefined} is true.
\begin{proof} The proof is similar to that of \cite[Thm. 4.1]{Scully1}, and begins with the following observation regarding the motivic decomposition of the quadric $Q$:

\begin{lemma} \label{lemmotivicsummand} In the situation of Theorem \ref{thmSPcondition}, $U(P)\lbrace \mathfrak{i}-1 \rbrace$ is isomorphic to a direct summand of $M(Q)$.
\begin{proof} By \cite[Thm. 4.15]{Vishik1}, it suffices to check that for every field extension $L$ of $F$, we have $\windex{p_L} > 0 \Leftrightarrow \windex{q_L} \geq \mathfrak{i}. $
The left to right implication is immediate; indeed, if $p_L$ is isotropic, then $L(p)$ is a purely transcendental extension of $L$, and so $\windex{q_L} = \windex{q_{L(p)}} \geq \windex{q_{F(p)}} = \mathfrak{i}$. For the other implication, assume that $\windex{q_L} \geq \mathfrak{i}$. If $(F_r)$ denotes the Knebusch splitting tower of $q$, then it follows from the very definition of $l$ that the free composite $F_l \cdot L$ is a purely transcendental extension of $L$ (compare the end of the proof of Proposition \ref{propdefined}). To show that $\windex{p_L} >0$, it therefore suffices to show that $p$ becomes isotropic over $F_l$. Note first that, by hypothesis, we have
\begin{eqnarray*} \mathrm{dim}(q_{l-1}) - \witti{1}{q_{l-1}} &=& \mathrm{dim}(q_{l-1}) - \frac{1}{2}\left(\mathrm{dim}(q_{l-1}) - \mathrm{dim}(q_l)\right) \\ &=&  (2^{s+1} - k) - \frac{1}{2}(2^{s+1} - 2k) = 2^s. \end{eqnarray*}
On the other hand, the form $q_{l-1}$ becomes isotropic over $F_{l-1}(p)$. Since $\mathrm{dim}(p) > 2^s$, and since $F_l =  F_{l-1}(q_{l-1})$, the desired assertion thus follows from Izhboldin's ``strong incompressibility'' theorem (\cite[Thm. 0.2]{Izhboldin1}).
 \end{proof} \end{lemma}

The proof of Theorem \ref{thmSPcondition} now proceeds by induction on $l$, with the case where $l = 1$ being trivial. Indeed, if $l=1$, then $\mathrm{dim}(q) = 2^{s+1} - k$ and $\witti{1}{q} = (\mathrm{dim}(q) -k)/2 = 2^s - k$. Assume now that $l\geq 2$ (in particular, $h(q)\geq 2$). Applying the induction hypothesis to the pair $(q_1,p_{F(q)})$, we immediately get that the integers $\mathrm{dim}(q_r)$ and $\witti{r+1}{q}$ have the prescribed form for all $1 \leq r < l$. In particular, we have $\mathrm{dim}(q_1) = 2^{s+1}a_1 + \epsilon_1$ for some $a_1 \geq 1$ and $-k \leq \epsilon_1 \leq k$. Since $\mathrm{dim}(q) = \mathrm{dim}(q_1) + 2\witti{1}{q}$, we then have
\begin{equation} \label{eq5.1} \mathfrak{i} = \frac{1}{2}(\mathrm{dim}(q) - k) = 2^sa_1 + \witti{1}{q} -\left(\frac{k - \epsilon_1}{2}\right). \end{equation}
It still remains show that $\mathrm{dim}(q)$ and $\witti{1}{q}$ are as claimed. We separate two cases:\\

\noindent {\it Case 1.} $\witti{1}{q} \leq (k - \epsilon_1)/2$. In this case, the assertions are clear. Indeed, we have 
\begin{eqnarray*} 2^{s+1}a_1 + \epsilon_1 = \mathrm{dim}(q_1) &<& \mathrm{dim}(q) \\
 &=& \mathrm{dim}(q_1) + 2\witti{1}{q} \\ 
 &\leq& \mathrm{dim}(q_1) + (k-\epsilon_1) = 2^{s+1}a_1 + k. \end{eqnarray*}
Since $-k \leq \epsilon_1 \leq k$, this shows that $\mathrm{dim}(q) = 2^{i+1}a_1 + \epsilon_0$ for some $-k < \epsilon_0 \leq k$, and we then have
$$ \witti{1}{q} = \frac{\mathrm{dim}(q) - \mathrm{dim}(q_1)}{2} = \frac{\epsilon_0 - \epsilon_1}{2} \leq \frac{k+\epsilon_0}{2}. $$

\noindent {\it Case 2.} $\witti{1}{q} > (k - \epsilon_1)/2$. In this case, let $u$ denote the smallest non-negative integer such that $2^u \geq \witti{1}{q}$. By Karpenko's theorem on the possible values of the first Witt index (\cite{Karpenko2}), the integer $\mathrm{dim}(q) - \witti{1}{q}$ is divisible by $2^u$. Since $\mathrm{dim}(q) = \mathrm{dim}(q_1) + 2\witti{1}{q}$, it follows that
\begin{equation} \label{eq5.2} 2^{s+1}a_1 + \epsilon_1 + \witti{1}{q} \equiv 0 \pmod{2^u}. \end{equation}
We claim that $u >i$. Suppose that this is not the case. Then \eqref{eq5.2} implies that $\witti{1}{q} = 2^u \mu - \epsilon_1$ for some integer $\mu$. Now, by hypothesis (and the fact that $\epsilon_1 \geq -k$), we have
$$ \witti{1}{q} + \epsilon_1 > \frac{k-\epsilon_1}{2} + \epsilon_1 = \frac{k + \epsilon_1}{2} \geq 0,$$  
and so $\mu > 0$. At the same time, we have $2^{u}\mu \leq 2^s$. Indeed, since $2^u \geq \witti{1}{q}$, the inequality $2^u\mu > 2^s$ would imply that 
$$ \epsilon_1 = 2^u\mu - \witti{1}{q} \geq 2^s, $$ 
which is not the case (since $\epsilon_1 \leq k < 2^s$). The situation is thus as follows: We have
\begin{equation} \label{eq5.3}  \mathrm{dim}(q) - \witti{1}{q} = \mathrm{dim}(q_1) + \witti{1}{q} = 2^{s+1}a_1 + 2^u \mu, \end{equation}
where $a_1 \geq 1$ and $0 < 2^u\mu \leq 2^s$. Now the integer $2^{s+1}a_1$ can be written in the form $2^{r_1} - 2^{r_2} + \cdots + 2^{r_{w-1}} - 2^{r_w}$ for some integers $r_1>r_2>\cdots>r_w \geq s+1$, while $2^u\mu$ can be written as $2^{r_{w+1}} - 2^{r_{w+2}} + \cdots + (-1)^{t-1}2^{r_t}$ for unique integers $s \geq r_{w+1} > r_{w+2} > \cdots > r_{t-1} > r_t+1 \geq 1$. Equation \ref{eq5.3} can therefore be re-written as
$$ \mathrm{dim}(q) - \witti{1}{q} = \underbrace{\left(2^{r_1} - 2^{r_2} + \cdots + 2^{r_{w-1}}- 2^{r_w}\right)}_{2^{s+1}a_1} + \underbrace{\left(2^{r_{w+1}} - 2^{r_{w+2}} + \cdots + (-1)^{t-1}2^{r_t}\right)}_{2^u\mu}. $$
Unless $2^u\mu = 2^s$ and $r_{w} = s+1$, this is precisely the presentation of $\mathrm{dim}(q) - \witti{1}{q}$ as an alternating sum of $2$-powers appearing in the statement of Vishik's Theorem \ref{thmexcellentconnections}. If $2^u\mu = 2^s$ and $r_{w} = s+1$, then the needed presentation is
$$ \mathrm{dim}(q) - \witti{1}{q} = \underbrace{\left(2^{r_1} - 2^{r_2} + \cdots + 2^{r_{w-1}}\right)}_{2^{s+1}(a_1 + 1)} - 2^{s}. $$
Either way, see that the Tate motive $\mathbb{Z} \lbrace 2^sa_1 \rbrace$ is isomorphic to a direct summand of $U(Q)_{\overline{F}}$. Indeed, this follows by applying Vishik's Theorem with $c=w+1$ in the first case, and with $c = w$ in the second. Let $j = \witti{1}{q} - 1 - (k - \epsilon_1)/2$. Since $\witti{1}{q} > (k - \epsilon_1)/2$, we then have $0 \leq j < \witti{1}{q}$. Thus, by \cite[Cor. 3.10]{Vishik4}, $U(Q)\lbrace j \rbrace$
is isomorphic to a direct summand of $M(Q)$. By the preceding discussion, $\mathbb{Z}\lbrace 2^sa_1 + j \rbrace$ is isomorphic to a direct summand of $U(Q)\lbrace j \rbrace_{\overline{F}}$. But 
$$ 2^sa_1 + j  = 2^sa_1 + \witti{1}{q} -(k- \epsilon_1)/2 -1 = \mathfrak{i} - 1$$ 
by equation \eqref{eq5.1}. In view of Lemma \ref{lemmotivicsummand}, we thus see that there are two indecomposable direct summands of $M(Q)$ containing the Tate motive $\mathbb{Z}\lbrace \mathfrak{i} - 1 \rbrace$ in their respective decompositions over $\overline{F}$, namely, $U(Q)\lbrace j \rbrace$ and $U(P) \lbrace \mathfrak{i} - 1 \rbrace$. As result, we must have $U(Q)\lbrace j \rbrace \cong U(P)\lbrace \mathfrak{i} - 1 \rbrace$ (see \S \ref{subsecmotives} above). In particular, we have $j = \mathfrak{i}-1$. 
Since $j< \mathfrak{i}_1(q)$, and since $\mathfrak{i} > 0$, this implies that $\mathfrak{i} = \witti{1}{q}$. In other words, it implies that $l=1$, which provides us with the needed contradiction to our supposition. We can therefore conclude that $u>s$, i.e., that $\witti{1}{q} > 2^s$. By \eqref{eq5.2}, it follows that $\witti{1}{q} = 2^{s+1}b_0 - \epsilon_1$ for some positive integer $b_0$. Moreover, we then have
$$\mathrm{dim}(q) = \mathrm{dim}(q_1) + 2\witti{1}{q} = 2^{s+1}(a_1 + 2b_0) - \epsilon_1. $$
Since $-k \leq \epsilon_1 \leq k$, this proves the theorem. \end{proof} \end{theorem}

Let us note that the proof of Theorem \ref{thmSPcondition} allows us to extract more information. Recall that a direct summand $N$ in the motive of a smooth projective $F$-quadric is said to be \emph{binary} if $N_{\overline{F}}$ is isomorphic to a direct sum of two Tate motives. We then have:

\begin{proposition} \label{propbinarymotive} If $k<2^{s-1} + 2^{s-2}$ and $\mathrm{dim}(q_{l-1}) = 2^{s+1} - k$, then the upper motive $U(P)$ of the quadric $P$ is binary.
\begin{proof} By Lemma \ref{lemmotivicsummand}, $U(P)\lbrace \mathfrak{i} - 1 \rbrace$ is isomorphic to a direct summand of $M(Q)$. Since $\mathfrak{i}-1 = \wittj{l}{q}-1$, it follows from \cite[Thm. 4.13]{Vishik4} that $N \coloneqq U(P) \lbrace \mathfrak{j}_{l-1}(q)\rbrace$ is also isomorphic to a direct summand of $M(Q)$. By an argument identical to that given in \cite[Proof of Thm. 7.7]{Vishik4}, it follows that $N$, and hence $U(P)$, is binary provided that $\witti{r}{q} < \witti{l}{q}$ for all $r>l$. But $\witti{l}{q} = (\mathrm{dim}(q_{l-1}) - k)/2 = 2^s- k> 2^{s-2}$ by hypothesis, and so it only remains to note that $\witti{r}{q} \leq 2^{s-2}$ for all $r >l$. Indeed, if $\mathrm{dim}(q_{r-1}) \leq 2^{s-1}$, then this is trivial; otherwise, our hypothesis again implies that $2^{s-1} < \mathrm{dim}(q_{r-1}) < 2^{s-1} + 2^{s-2}$ and so the claim follows from Hoffmann's Separation theorem (see \cite[Cor. 1]{Hoffmann1}). \end{proof} \end{proposition}

As conjectured by Vishik (cf. \cite[Conj. 4.21]{Vishik4}), the upper motive of an anisotropic projective quadric over a field of characteristic $\neq 2$ should be binary only when its underlying quadratic form is a Pfister neighbour (the converse being a well-known result of Rost (\cite[Prop. 4]{Rost})). Significant evidence for the validity of this assertion has been given in \cite[\S 6]{IzhboldinVishik}. Moreover, in \cite{Karpenko1}, Karpenko showed that Vishik's conjecture holds for forms of dimension $\leq 16$. Applying this to the special case where $k = 0$ and $\mathrm{deg}(q) = s+1$, we get the following result promised in \S \ref{subsecmainhyp} above:

\begin{corollary} \label{corbinarymotive} If $q_{F(p)}$ is hyperbolic and $\mathrm{deg}(q) = s+1$, then: 
\begin{enumerate} \item The upper motive $U(P)$ of the quadric $P$ is binary. 
\item If $s \leq 3$ $($i.e., $\mathrm{dim}(p) \leq 16)$, there exists an $(s+1)$-fold Pfister form $\pi$ and an odd-dimensional form $r$ such that $q \simeq \pi \otimes r$. Moreover, $p$ is a neighbour of $\pi$. \end{enumerate}
\begin{proof} The first statement is the $k=0$ case of Proposition \ref{propbinarymotive}. The second statement then follows from the aforementioned result of Karpenko together with \cite[Cor. 23.6]{EKM}. More precisely, Karpenko showed in \cite{Karpenko1} that if the upper motive of $P$ is binary, and $\mathrm{dim}(p) \in \lbrace 3,5,9 \rbrace$, then $p$ is a Pfister neighbour. But this implies the same assertion in all dimensions $\leq 16$ by \cite[Thm. 6.1]{IzhboldinVishik} and \cite[Cor. 3.9, Cor. 4.7]{Vishik1}. \end{proof} \end{corollary}

We can now give the proof of Theorem \ref{thmadditional}:

\begin{proof}[Proof of Theorem \ref{thmadditional}] We may assume that $l \geq 1$, since otherwise there is nothing to prove. If $\mathrm{dim}(q_{l-1}) = 2^{s+1} - k$, then Conjecture \ref{conjrefined} holds by Theorem \ref{thmSPcondition}. Thus, by Lemma \ref{lemdimatprevious}, we can assume that $\mathrm{dim}(q_{l-1}) \geq 2^{s+2} - k$. Now, if $\mathrm{dim}(q) \leq 2^{s+2} + k$, then we obviously have $\mathrm{dim}(q_{l-1}) = 2^{s+2} -k$, and the assertions of Conjecture \ref{conjrefined} hold trivially. This takes care of case (2) of the theorem. It now remains to treat the case where $k \leq 7$. As $\mathrm{dim}(q_{l-1}) \geq 2^{s+2} -k$, Proposition \ref{propdefined} and Lemma \ref{lemdefinedimpliesconj} show that it suffices to know that Kahn's Conjecture \ref{conjKahn} holds in the following cases:
\begin{enumerate} \item[(i)] $\mathrm{dim}(\tau) \in \lbrace 0,1 \rbrace$, $\mathrm{dim}(\varphi) \geq 7$.
\item[(ii)] $\mathrm{dim}(\tau) \in \lbrace 2,3 \rbrace$, $\mathrm{dim}(\varphi) \geq 13$.
\item[(iii)] $\mathrm{dim}(\tau) \in \lbrace 4,5,6,7 \rbrace$, $\mathrm{dim}(\varphi) \geq 25$. \end{enumerate}
But all these cases of Kahn's conjecture are already known to be true: for the cases where $\mathrm{dim}(\tau) \leq 5$, see \cite[Thm. 2]{Kahn1}; for the case where $\mathrm{dim}(\tau) = 6$, see \cite[Thm. principale]{Laghribi1}, and for the case where $\mathrm{dim}(\tau) = 7$, see \cite[Thm. 3.9]{IzhboldinVishik}. \end{proof}

\begin{corollary} Conjecture \ref{conjmain} holds if $\mydim{q} \leq 2^{s+1} + 2^{s-1}$. 
\begin{proof} If $k \geq 2^{s-1}$, apply \ref{thmadditional}; otherwise, apply Theorem \ref{thmmain}. \end{proof} \end{corollary}

\begin{remarks} \label{remsfinal} (1) Laghribi (\cite[Th\'{e}or\`{e}me principale]{Laghribi1}) has shown that if $\mathrm{dim}(\varphi) > 16$, then Conjecture \ref{conjKahn} also holds for some special classes of $8$, $9$ and $10$-dimensional forms $\tau$. This permits to extend Theorem \ref{thmadditional} to the case where $k \in \lbrace 8,9,10 \rbrace$ under some additional assumptions on $q$. We refrain from going into the details here.

(2) Conjecturally, the form $\anispart{(q_{F(p)})}$ should be defined over $F$ in the situation of Conjecture \ref{conjrefined} provided that \emph{either} (i) $k < 2^{s-1} + 2^{s-2}$ or (ii) $l \geq 1$ and $\mathrm{dim}(q_{l-1}) \neq 2^{s+1} - k$. Indeed, that the second condition should be enough is implicit in our proof of Theorem \ref{thmadditional}. As for the first condition, we may additionally assume in this case that $l \geq 1$ and $\mathrm{dim}(q_{l-1}) = 2^{s+1} - k$. By Proposition \ref{propbinarymotive}, this ensures that the upper motive of $P$ is binary. According to \cite[Conj. 4.21]{Vishik1}, this should in turn imply that $p$ is a Pfister neighbour. By replacing $p$ by its ambient Pfister form, we could then assume that $\mathrm{dim}(p) = 2^{s+1} > 2k$, and so the claim would follow directly from Conjecture \ref{conjKahn}. \end{remarks}

\noindent {\bf Acknowledgements.} I would like to thank Alexander Vishik for helpful comments. The support of a PIMS postdoctoral fellowship and NSERC discovery grant during the preparation of this article is gratefully acknowledged.

\bibliographystyle{alphaurl}

\end{document}